\documentclass[10pt]{article}

\usepackage{t1enc}
\usepackage[latin1]{inputenc}
\usepackage[english]{babel}
\usepackage{amsmath,amsthm}
\usepackage{amsfonts}
\usepackage{latexsym}
\usepackage{graphicx}
\usepackage{txfonts}

\newtheorem{thm}{Theorem}
\newtheorem{lem}[thm]{Lemma}
\newtheorem{cor}[thm]{Corollary}

\newtheorem{prob}[thm]{Problem}

\date{}

\begin{document}

\title{Total coloring of pseudo-outerplanar graphs\footnotetext{Email addresses: sdu.zhang@yahoo.com.cn, gzliu@sdu.edu.cn.}\thanks{Supported in part by NSFC(10971121, 11101243, 61070230), RFDP(20100131120017) and GIIFSDU(yzc10040).}}
\author{Xin Zhang,~Guizhen Liu\thanks {Corresponding author.}\\
{\small School of Mathematics, Shandong University, Jinan 250100, China}}
\date{}

\maketitle

\begin{abstract}
A graph is pseudo-outerplanar if each of its blocks has an embedding in the plane so that the vertices lie on a fixed circle and the edges lie inside the disk of this circle with each of them crossing at most one another. In this paper, the total coloring conjecture is completely confirmed for pseudo-outerplanar graphs. In particular, it is proved that the total chromatic number of every pseudo-outerplanar graph with maximum degree $\Delta\geq 5$ is $\Delta+1$.\\[.2em]
\end{abstract}

\section{Introduction}

A total coloring of a graph $G$ is an assignment of colors to the vertices and edges of $G$ such that every pair of adjacent/incident elements receive different colors. A $k$-total coloring of a graph $G$ is a total coloring of $G$ from a set of $k$ colors. The minimum positive integer $k$ for which $G$ has a $k$-total coloring, denoted by $\chi''(G)$, is called the total chromatic number of $G$. It is easy to see that $\chi''(G)\geq \Delta(G)+1$ for any graph $G$ by looking at the color of a vertex with maximum degree and its incident edges. The next step is to look for a Brooks-typed or Vizing-typed upper bound on the total chromatic number in terms of maximum degree. It turns out that the total coloring version of maximum degree upper bound is a difficult problem and has eluded mathematicians for nearly 50 years. The most well-known speculation is the total coloring conjecture, independently raised by Behzad \cite{Behzad.1965} and Vizing \cite{Vizing.1968}, which asserts that every graph of maximum degree $\Delta$ admits a $(\Delta+2)$-total coloring. This conjecture remains open, however, many beautiful results concerning it have been obtained (cf.~\cite{yap}). In particular, the total chromatic number of all outerplanar graphs has been determined completely by Zhang et al.\ \cite{ZZF} and that of all series-parallel graphs \cite{Wu} has been determined completely by Wu and Hu \cite{ZZF}.

A graph is pseudo-outerplanar if each of its blocks has an embedding in the plane so that the vertices lie on a fixed circle and the edges lie inside the disk of this circle with each of them crossing at most one another. For example, $K_{2,3}$ and $K_4$ are both pseudo-outerplanar graphs. This notion was introduced by Zhang, Liu and Wu in \cite{ZPOPG}, where the edge-decomposition of pseudo-outerplanar graphs into forests with a specified property was studied. In this paper, we prove that the total chromatic number of every pseudo-outerplanar graph with maximum degree $\Delta\geq 5$ is exactly $\Delta+1$ and thus the total coloring conjecture holds for all pseudo-outerplanar graphs.

\section{Main results and their proofs}

To begin with, let us review an useful structural property of pseudo-outerplanar graphs which was proved in \cite{ZPOPG}.

\begin{lem}\label{struc}
Let $G$ be a pseudo-outerplanar graph with minimum degree at least two. Then

\noindent $(1)$ $G$ has an edge $uv$ such that $d(u)=2$ and $d(v)\leq 4$, or

\noindent$(2)$ $G$ has a $4$-cycle $uxvy$ such that $d(u)=d(v)=2$, or

\noindent$(3)$ $G$ has a $4$-cycle $uxvy$ such that $d(u)=d(v)=3$ and $uv\in E(G)$, or

\noindent$(4)$ $G$ has a $7$-path $x'uxvywy'$ such that $d(u)=d(v)=d(w)=2$, $d(x)=d(y)=5$ and $xx', yy', xy, x'y, xy'\in E(G)$.
\end{lem}

Here it should be remarked that Lemma \ref{struc} is a straightforward simplification of the corresponding theorem in \cite{ZPOPG} (which had more subcases).

\begin{thm}
Let $G$ be a pseudo-outerplanar graph with maximum degree $\Delta$ and let $M$ be an integer such that $\Delta\leq M$. If $M\geq 5$, then $\chi''(G)\leq M+1$.
\end{thm}

\begin{proof}
We shall prove the theorem by induction on $|V(G)|+|E(G)|$. So we assume that $G$ is 2-connected and thus $\delta(G)\geq 2$. Set $U=\{1,\cdots,M+1\}$. In the next, we complete the proof by verifying the following four claims; they imply a contradiction to Lemma \ref{struc}.\vspace{1mm}

\noindent \emph{Claim $1$. If $uv\in E(G)$ and $d(u)=2$, then $d(v)\geq M$.}\vspace{1mm}

Suppose, to the contrary, that $d(v)\leq M-1$. Consider the graph $G'=G-uv$. By induction, $G'$ has an $(M+1)$-total coloring $c$. Denote the other neighbor of $u$ by $w$. Now erase the color on $u$ and color the edge $uv$ with $c(uv)\in A(uv):=U\setminus (\{c(v),c(uw)\} \cup \{c(vx)~|~vx\in E(G')\})$. This is possible since $|A(uv)|\geq M+1-(2+M-2)=1$. Denote the extended coloring at this stage still by $c$. Then at last we color $u$ with $c(u)\in U\setminus \{c(v),c(w),c(uv),c(uw)\}$, which is also possible since $|U|\geq 6$.\vspace{1mm}

\noindent \emph{Claim $2$. $G$ does not contain a $4$-cycle $uxvy$ such that $d(u)=d(v)=2$.}\vspace{1mm}

Otherwise, $G'=G-\{u,v\}$ admits an $(M+1)$-total coloring by induction and every edge of the 4-cycle has at least two available colors since it is incident with at most $\Delta-1$ colored elements. This implies that one can extend the coloring of $G'$ to the four edges $ux,vx,uy$ and $vy$ since every 4-cycle is 2-edge-choosable. At last, the two vertices $u$ and $v$ can be easily colored since they are both of degree two.\vspace{1mm}

\noindent \emph{Claim $3$. $G$ does not contain a $4$-cycle $uxvy$ such that $d(u)=d(v)=3$ and $uv\in E(G)$.}\vspace{1mm}

Suppose, to the contrary, that $G$ contains such a 4-cycle. We consider the graph $G'=G-uv$ that has an $(M+1)$-total coloring $c$ by induction. One can find that the only obstacle of extending $c$ to $uv$ is the case when $M=5$ and $U=\{c(u),c(v),c(ux),c(uy),c(vx),c(vy)\}$. Without loss of generality, let $c(ux)=1,c(uy)=2,c(vx)=3,c(vy)=4,c(u)=5$ and $c(v)=6$. If $c(x)\neq 4$, then we recolor $u$ by 4 (note that $c(y)\neq 4$) and then color $uv$ by 5. So we assume $c(x)=4$. Similarly we can prove $c(y)=3$. Therefore, we can recolor $v$ by 1 and then color $uv$ by 6, a contradiction.\vspace{1mm}

\noindent \emph{Claim $4$. $G$ does not contain a $7$-path $x'uxvywy'$ such that $d(u)=d(v)=d(w)=2$, $d(x)=d(y)=5$ and $xx', yy', xy, x'y, xy'\in E(G)$.}\vspace{1mm}

Suppose, to the contrary, that $G$ contains such a 7-path. We consider the graph $G'=G-\{u,v,w\}$, which admits an $(M+1)$-total coloring $c$ by induction. Denote by $A(e)$ the set of available colors to properly color an edge $e\in \{ux,vx,vy,wy,ux',wy'\}$. It is easy to see that $|A(ux')|=|A(wy')|=1$ and $|A(ux)|=|A(vx)|=|A(vy)|=|A(wy')|=2$, moreover, $A(ux)=A(vx)$ and $A(vy)=A(wy)$. If $A(ux')\neq A(wy')$, without loss of generality assume that $A(ux')=\{1\}$ and $A(wy')=\{2\}$, then color $ux'$ with 1 and $wy'$ with $2$. If $1\in A(ux)$, then color $ux$ with $c(ux)\in A(ux)\setminus \{1\}$ and $vx$ with 1. If $1\not\in A(ux)$, then color $vx$ with $c(vx)\in A(vx)\setminus \{2\}$ and $ux$ with $c(ux)\in A(ux)\setminus \{c(vx)\}$. In each case we have $c(vx)\neq c(wy')$. Thus we can color $vy$ with $c(vy)\in A(vy)\setminus \{c(vx)\}$ and $wy$ with $c(wy)\in A(wy)\setminus \{c(wy')\}$ such that $c(vy)\neq c(wy)$. At this stage, the three vertices $u,v$ and $w$ can be easily colored since they are all of degree two. So we assume that $A(ux')=A(wy')=\{1\}$. Now we firstly color $ux'$ and $wy'$ by 1. If $1\not\in A(ux)$, then color $vx$ with $c(vx)\in A(vx)$ and $ux$ with $c(ux)\in A(ux)\setminus \{c(vx)\}$. The current extended coloring satisfies that $c(vx)\neq c(wy')$. Therefore, we can color the remaining elements similarly as before. So we assume that $1\in A(ux)$. This implies that $1\not\in \{c(x),c(xy),c(xx'),c(xy')\}$ and thus we can exchange the colors on $ux'$ and $xx'$. By doing so we obtain a new coloring satisfying $c(ux')\neq c(wy')$ and therefore we can extend this partial coloring to $G$ by a same argument as above.
\end{proof}

\begin{cor}\label{main}
Every pseudo-outerplanar graph with maximum degree $\Delta\geq 5$ is $(\Delta+1)$-total colorable.
\end{cor}

Note that every graph with maximum degree $\Delta\leq 5$ is $(\Delta+2)$-total colorable (see \cite{Kostochka.1996} and Chapter 4 of \cite{yap}). So we also have the following corollary.

\begin{cor}
The total coloring conjecture holds for all pseudo-outerplanar graphs.
\end{cor}

The graph (a) in Figure \ref{fig} is a pseudo-outerplanar graph, since it has a pseudo-outerplanar drawing as (b). One can easy to check that it is a graph with maximum degree 3 and total chromatic number 5. Thus the upper bound for $\Delta$ in Corollary \ref{main}, although probably not the best possible, cannot be less than 4.
\begin{figure}
\begin{center}
  \includegraphics[width=7.5cm,height=2.8cm]{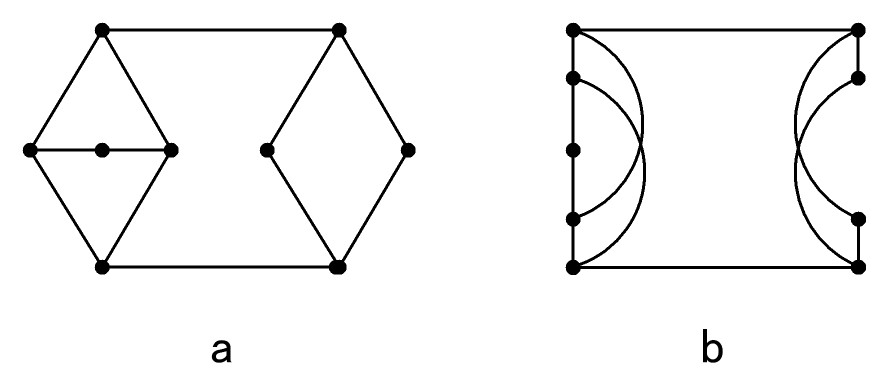}\\
  \caption{A pseudo-outerplanar graph $G$ with $\Delta(G)=3$ and $\chi''(G)=5$}\label{fig}
\end{center}
\end{figure}
At last, we leave the following open problem to end this paper.

\begin{prob}
To determine the total chromatic number of pseudo-outerplanar graphs with maximum degree four.
\end{prob}



\begin{thebibliography}{10}\setlength{\itemsep}{0pt}

\bibitem{Behzad.1965} M. Behzad. Graphs and their chromatic numbers. Doctoral thesis, Michigan State University, 1965.


\bibitem{Kostochka.1996} A. V. Kostochka. The total chromatic number of any multigraph with maximum degree five
is at most seven. Discrete Mathematics. 162 (1996) 199--214.

\bibitem{Vizing.1968} V. Vizing. Some unsolved problems in graph theory. Uspekhi Mat. Nauk, 23 (1968) 117--134.

\bibitem{Wu} J. L. Wu, D. Hu. Total coloring of series-parallel graphs. Ars Comb. 73 (2004) 209--211.

\bibitem{yap}  H. P. Yap. Total colourings of graphs. Lecture Notes in Mathematics 1623, Berlin; London: Springer (1996).

\bibitem{ZPOPG} X. Zhang, G. Liu, J. L. Wu. Edge covering pseudo-outerplanar graphs with forests. arXiv:1108.3877v1 [math.CO].

\bibitem{ZZF} Z. Zhang, J. Zhang, J. Wang. The total chromatic number of some graphs. Scientia Sinica A 31 (1988) 1434--1441


\end{thebibliography}
\end{document}